\documentclass{amsart}

\usepackage{amsthm,amssymb} 
\usepackage{graphicx} 

\theoremstyle{plain}
\newtheorem{thm}{Theorem}
\newtheorem{prop}[thm]{Proposition}
\newtheorem{lem}[thm]{Lemma}
\newtheorem{coro}[thm]{Corollary}

\newtheorem*{thm4}{Theorem 4}

\theoremstyle{definition}
\newtheorem{defn}[thm]{Definition}

\newtheorem*{exa}{Example}

\theoremstyle{remark}
\newtheorem*{rem}{Remark}
\newtheorem*{claim*}{Claim}
\newtheorem{case}{Case}

\newcommand{\nil}{\varnothing}
\newcommand{\bfA}{\mathbf{A}}
\newcommand{\bfD}{\mathbf{D}}
\newcommand{\bfE}{\mathbf{E}}
\newcommand{\bfQ}{\mathbf{Q}}

\begin{document} 
%

%
%

\title[Fox reimbedding and Bing submanifolds]{Fox reimbedding and Bing submanifolds}
\author[K. Nakamura]{Kei Nakamura}
\address{Department of Mathematics\\ Temple University\\ Philadelphia, PA 19122}
\email{nakamura@temple.edu}

\begin{abstract} 
Let $M$ be an orientable closed connected $3$-manifold. We introduce the notion of \emph{amalgamated Heegaard genus} of $M$ with respect to a closed separating $2$-manifold $F$, and use it to show that the following two statements are equivalent: (i) a compact connected 3-manifold $Y$ can be embedded in $M$ so that the exterior of the image of $Y$ is a union of handlebodies; and (ii) a compact connected $3$-manifold $Y$ can be embedded in $M$ so that every knot in $M$ can be isotoped to lie within the image of $Y$.
\par Our result can be regarded as a common generalization of the reimbedding theorem by Fox \cite{Fox:Reimbedding} and the characterization of $3$-sphere by Bing \cite{Bing:S3}, as well as more recent results of Hass and Thompson \cite{Hass--Thompson:Bing} and Kobayashi and Nishi \cite{Kobayashi--Nishi:Bing}.
\end{abstract}

\maketitle

%
%

%
\section{Introduction} \label{Introduction}

This paper presents a common generalization of two well-known classical theorems regarding the topology of 3-manifolds. The first is a theorem of Fox, published in 1948, which is often referred to as  Fox reimbedding theorem in the modern literature:

\begin{thm}[\cite{Fox:Reimbedding}] \label{Fox}
Every compact connected $3$-submanifold $Y$ of the $3$-sphere can be reimbedded in the $3$-sphere so that the exterior of the image of $Y$ is a union of handlebodies, i.e. regular neighborhoods of embedded graphs.
\end{thm}

The second is a theorem of Bing, published in 1958, which gives a characterization of the $3$-sphere:

\begin{thm}[\cite{Bing:S3}] \label{Bing}
A closed connected $3$-manifold $M$ is homeomorphic to the $3$-sphere if and only if every knot in $M$ can be isotoped to lie within a closed $3$-ball.
\end{thm}

Fox reimbedding theorem was extended by Scharlemann and Thompson in \cite{Scharlemann--Thompson:Fox} where they proved a reimbedding theorem for a submanifold of irreducible non-Haken $3$-manifolds and also refined the reimbedding procedure so that the reimbedded submanifold and its exterior are aligned in a suitable sense; see also \cite{Menasco--Thompson:Fox} and \cite{Scharlemann:Fox} for related results.

Bing's theorem was followed by a few generalizations where the assumptions were weakened in a hope to prove Poincar\'{e} Conjecture; see \cite{McMillan:HomologicallyTrivial}, \cite{Myers:OpenBook}, \cite{Myers:SimpleKnots}, and \cite{Gordon--Montesinos:FiberedKnots} for example. In another direction, Hass and Thompson gave an analogous characterization of lens spaces in \cite{Hass--Thompson:Bing}: a closed connected orientable $3$-manifold $M$ is homeomorphic to a lens space (possibly $S^3$ or $S^2 \times S^1$) if and only if there exists a solid torus $V$ in $M$ such that every knot in $M$ can be isotoped to lie within $V$. Along this latter direction, Kobayashi and Nishi gave a further generalization in \cite{Kobayashi--Nishi:Bing}, settling the conjecture of Hass and Thompson affirmatively: a closed connected orientable $3$-manifold $M$ admits a genus $g$ Heegaard splitting if and only if there exists a genus $g$ handlebody $V$ in $M$ such that every knot in $M$ can be isotoped to lie within $V$. 

The properties of submanifolds that played crucial roles in the theorems above can be studied in a more general context of an arbitrary closed connected orientable $3$-manifold $M$ and a compact connected $3$-submanifold $Y$ of $M$.

\begin{defn}
Let $M$ be a closed connected oriented 3-manifold, and $Y$ be a compact connected submanifold of $M$.
\begin{itemize}
\item $Y$ is said to be a \emph{Fox submanifold} of $M$ if its exterior is a union of handlebodies. A reimbedding of $Y$ into $M$ is said be a \emph{Fox reimbedding} if the image of $Y$ is a Fox submanifold of $M$.
\item $Y$ is said to be a \emph{Bing submanifold} of $M$ if every knot in $M$ can be isotoped to lie within $Y$.
\end{itemize}
\end{defn}

Any Fox submanifold is always a Bing submanifold, while a Bing submanifold is usually not a Fox submanifold. Our main result is the following:

\begin{thm} \label{main theorem}
For any closed connected $3$-manifold $M$, every Bing submanifold of $M$ admits a Fox reimbedding into $M$; hence, a compact connected 3-manifold can be embedded in $M$ as a Fox submanifold if and only if it can be embedded in $M$ as a Bing submanifold.
\end{thm}

With suitable choices of $M$ and $Y$, Theorem~\ref{main theorem} indeed specializes to the reimbedding theorem of Fox \cite{Fox:Reimbedding}, the characterization of the $3$-sphere by Bing \cite{Bing:S3}, the characterization of the lens spaces by Hass and Thompson \cite{Hass--Thompson:Bing}, and the characterization of manifolds admitting a genus $g$ Heegaard splitting by Kobayashi and Nishi \cite{Kobayashi--Nishi:Bing}. We also give a new characterization of manifolds admitting a genus $g$ Heegaard splitting and of manifolds admitting a cross-cap genus $g$ one-sided Heegaard splitting.

A reader may notice that the results in \cite{Bing:S3} and \cite{Hass--Thompson:Bing} predate the affirmative resolution of the \emph{Geometrization Conjecture}, from which one obtains much stronger characterizations of 3-sphere and lens spaces in terms of \emph{homotopy of loops}. It turns out that  Theorem~\ref{main theorem} cannot be strengthened by allowing homotopy of loops instead of isotopy of knots; we will explain examples that illustrate the existence of a closed $3$-manifold $M$ and a compact submanifolds $Y$ such that every loop in $M$ can be homotoped to lie within $Y$ while $Y$ admit no Fox reimbedding. The discussion of these examples leads us to a characterization of manifolds that are counterexamples to Waldhausen's question on the rank and the genus of $3$-manifolds.

\subsection*{Outline}

We review basic notions of $3$-manifolds in \S\ref{Preliminaries} and collect some facts about irreducible knots and simple knots in \S\ref{Irreducible Knots and Simple Knots}. Next, we study the genera of Heegaard splittings obtained by a process known as \emph{amalgamation} in \S\ref{Amalgamated Heegaard Genus}; the notion of \emph{amalgamated Heegaard genus} introduced in this section plays an essential role in the proof of Theorem~\ref{main theorem}. The main body of the proof is given in \S\ref{Characteristic Knots} as propositions involving the \emph{characteristic knots}; in \S\ref{Main Theorem and Examples}, we complete the proof and elaborate on its consequences mentioned in this introduction.

\subsection*{Acknowledgement}

The author would like to thank Joel Hass for providing insightful comments. The author would also like to thank Yoav Rieck for helpful discussions and Yuya Koda for carefully reading the earlier exposition of this work.

\section{Preliminaries} \label{Preliminaries}

\subsection{Conventions and Notations}

Throughout this article, we work with the piecewise-linear category, and all embeddings are assumed to be locally flat. 

A connected component of a manifold (or, more generally, a complex) $M$ will simply be called a \emph{component}, and we write $|M|$ for the number of the components.
For a submanifold (or a subcomplex) $X$ of a manifold (or a complex) $M$, we write $N(X,M)$ and $\overline{N}(X,M)$ respectively for an open regular neighborhood and a closed regular neighborhood of $X$ with respect to the topology of $M$, and we write $E(X,M):=M-N(X,M)$ for the \emph{exterior} of $X$ in $M$. When the ambient space $M$ is clear from the context, we may simply write $N(X)$, $\overline{N}(X)$, and $E(X)$ respectively for $N(X,M)$, $\overline{N}(X,M)$, and $E(X,M)$.

Any $2$-manifold will be called a \emph{surface}; they may or may not be connected.
A simple closed curve on a closed surface $F$ is said to be \emph{essential} if it does not bound a disk on $F$, and it is said to be \emph{trivial} otherwise.
If $F$ is a closed connected orientable surface, the \emph{genus} of $F$ will be denoted by $g(F)$ as usual. If $F$ is a closed disconnected orientable surface, we write $g(F)$ and $g_{\max}(F)$ for the \emph{total genus} and the \emph{maximal component-wise genus} respectively; namely, if $F$ consists of components $F_i$, $i \in \mathcal{I}$, then we set $g(F):=\sum_{i \in \mathcal{I}} g(F_i)$ and $g_{\max}(F):=\max_{i \in \mathcal{I}} g(F_i)$.

Every $3$-manifold will be orientable throughout this article.
A disjoint union of circles embedded in a $3$-manifold $M$ will be called a \emph{link} in $M$, and a one-component link is called a \emph{knot} in $M$. The \emph{exterior} of a link $L$ in $M$ is denoted by $E(L)$ or $E(L,M)$ as usual and in accords with our notation introduced above.
Whenever we speak of surfaces in $3$-manifolds, they are assumed to be embedded unless otherwise stated.
If a surface $F$ \emph{separates} $M$, then $F$ gives rise to a decomposition of $M$ into two compact submanifolds $Y_1$ and $Y_2$ such that $M=Y_1 \cup Y_2$ and $Y_1 \cap Y_2=F$; for brevity, such a decomposition will be denoted by $M=Y_1 \cup_F Y_2$. 
An orientable surface $F$ in a $3$-manifold $M$ is said to be \emph{compressible} if either (i) there exists a component of $F$ that bounds a closed $3$-ball in $M$, or (ii) there exists an embedded disk $D$ in $M$, called a \emph{compression disk} for $F$, such that $\partial D$ is an essential simple closed curve on $F$ and $D \cap F=\partial D$; $F$ is said to be \emph{incompressible} otherwise. A $3$-manifold $M$ is said to be \emph{irreducible} if every $2$-sphere in $M$ is compressible, and it is said to be \emph{atoroidal} if every incompressible $2$-torus in $M$ is $\partial$-parallel.

\subsection{Handlebodies and Compression bodies}

A \emph{handlebody} is a connected compact 3-manifold which is homeomorphic to the closed regular neighborhood of a connected graph embedded in an orientable 3-manifold, e.g. in $S^3$.
A \emph{compression body} is a generalization of a handlebody. In this paper, we use the definition of a compression body that appears in \cite{Hayashi--Shimokawa:ThinPosition} and \cite{Saito--Scharlemann--Schultens:GeneralizedHeeg}. Namely, a \emph{compression body} $V$ is a connected $3$-manifold obtained from a closed $3$-ball $B$ or $F \times I$, where $F$ is a (possibly disconnected) closed orientable surface and $I=[0,1]$, by attaching some number (possibly 0) of $1$-handles on $\partial B$ or $F \times \{1\}$ respectively.

The boundary $\partial V$ of a compression body $V$ is subdivided into two subsets, $\partial_-V$ and $\partial_+V$, as follows. Let (i) $\partial_-V:=\nil$ if $V$ is constructed from a closed $3$-ball $B$, and (ii) $\partial_-V:=F \times \{0\}$ if $V$ is constructed from $F \times I$; then, set $\partial_+V:=\partial V-\partial_-V$. By construction, $\partial_+V$ is always non-empty; moreover, since we require $V$ to be connected, it also follows from the construction that $\partial_+V$ is always connected. When $V$ is constructed from a closed $3$-ball $B$, the compression body $V$ is actually a handlebody and $\partial_+V=\partial V$; indeed, every handlebody can be constructed in this manner. Let us remark that, in the definition above, $\partial_+V$ is allowed to be a sphere and $\partial_-V$ is allowed to have a sphere component; this is in contrast with an alternative definition of a compression body that has appeared in the literature, where $\partial V$ must not have a sphere components.

A \emph{spine} of a compression body $V$ is a subcomplex $\mathbf{\Sigma}_V=\bfA \cup \partial_-V$, where $\bfA$ is an embedded graph, such that $E(\mathbf{\Sigma}_V,V) \cong \partial_+V \times I$. A spine $\mathbf{\Sigma}_V$ is said to be \emph{minimal} if (i) $V$ is a handlebody and $\mathbf{\Sigma}_V$ is a bouquet of circles, or (ii) $V$ is not a handlebody and $\mathbf{\Sigma}_V$ is a collection of properly embedded arcs in $V$.
A \emph{disk system} $\bfD \subset V$ of a compression body $V$ is a pairwise disjoint and pairwise non-parallel (possibly empty) collection of compression disks for $\partial_+V$ such that $E(\bfD,V)$ is homeomorphic to a union of $\partial_-V \times I$ and some closed $3$-balls. A disk system $\bfD$ is said to be \emph{minimal} if (i) $V$ is a handlebody and $E(\bfD,V) \cong B^3$, or (ii) $V$ is not a handlebody and $E(\bfD,V) \cong \partial_-V \times I$.

By the Poincar\'{e}-Lefschez duality of handle decompositions, we have a one-to-one correspondence between spines and disk systems of a compression body $V$; in particular, minimal spines and minimal disk systems are related bijectively by this duality. Given a minimal disk system $\bfD$ of a compression body $V$, we can regard $\overline{N}(\bfD,V)$ as a collection of 1-handles with $\bfD$ as the cocore. If $V$ is a handlebody, we obtain a minimal spine by extending the core arcs of these 1-handles with the radial arcs in $E(\bfD,V) \cong B^3$. If $V$ is not a handlebody, we obtain a minimal spine by extending the core arcs with vertical arcs in $E(\bfD,V) \cong \partial_-V \times I$. In either case, the spine obtained from $\bfD$ as above is called the \emph{dual spine of $\bfD$}. Reversing the construction, we see that each minimal spine $\mathbf{\Sigma}_V$ gives rise to a minimal disk system, called the \emph{dual disk system of $\mathbf{\Sigma}_V$}.

\subsection{Heegaard Splittings} \label{Heegaard Splittings}

Let $M$ be a connected compact $3$-manifold, and let $\partial M=\partial_V M \sqcup \partial_W M$ be a bipartition of $\partial M$. A decomposition $M=V \cup_S W$ is a \emph{Heegaard splitting} of $(M, \partial_V M, \partial_W M)$ if $S$ decomposes $M$ into compression bodies $V$ and $W$ such that $S=V \cap W = \partial_+V=\partial_+W$, $\partial_-V=\partial_V M$ and $\partial_-W=\partial_W M$; the surface $S$ is the \emph{Heegaard surface} of $(M, \partial_V M, \partial_W M)$ for this splitting. When there is no need to specify the bipartition of $\partial M$, we may simply say that $M=V \cup_S W$ is a Heegaard splitting of $M$, and that $S$ is a Heegaard surface for $M$.
%
A knot $K$ in $M$ is said to be in a \emph{bridge position with respect to a Heegaard splitting $M=V \cup_S W$} if $K$ meets each compression body in $\partial$-parallel arcs. 

The genus of a Heegaard splitting $M=V \cup_S W$ is defined to be the genus of the Heegaard surface $S$. The \emph{Heegaard genus of $M$}, or simply the \emph{genus of $M$}, is defined to be the minimum of genera of Heegaard splittings of $M$, and it is denoted by $g(M)$. We can also define a more restricted notion of Heegaard genus by specifying the bipartition $\partial M=\partial_V M \sqcup \partial_W M$. The \emph{genus of $(M, \partial_V M, \partial_W M)$} is defined to be the minimum of genera of Heegaard splittin gs of $(M, \partial_V M, \partial_W M)$, and it is denoted by $g(M, \partial_V M, \partial_W M)$.

An extreme choice of a bipartition of $\partial M$ is given by $\partial_V M=\partial M$ and $\partial_W M=\nil$. Following Kobayashi and Nishi \cite{Kobayashi--Nishi:Bing}, we call the Heegaard splitting of $(M,\partial M,\nil)$ a \emph{tunnel-type} Heegaard splitting of $M$. The \emph{tunnel-type genus} of $M$ is defined to be $g^t(M):=g(M,\partial M, \nil)$, i.e. the minimum of genera of tunnel-type Heegaard splittings of $M$. When $|\partial M| \leq 1$, there is a unique bipartition $\partial M=\partial M \sqcup \nil$; hence, every Heegaard splitting of $M$ is of tunnel-type, and we have $g^t(M)=g(M)$.

For a $3$-manifold $M$ that is not connected, a Heegaard splitting of $M$ is defined to be a union of Heegaard splittings of components of $M$. In other words, we have a Heegaard splitting $M_i=V_i \cup_{S_i} W_i$ for each component $M_i$, and a Heegaard splitting of $M$ is the decomposition $M=V \cup_S W$ where $V=\bigcup_i V_i$, $W=\bigcup_i W_i$, and $F=\bigcup_i S_i$. The genus of the splitting is defined to be $g(S)$, i.e. the total genus of $F$. The tunnel-type splittings and tunnel-type genus are defined analogously for disconnected $M$ as well by requiring $\partial_-V=\partial M$ and $\partial_-W=\nil$.

\subsection{Tunnel Number} \label{Tunnel Number}

Given a connected compact $3$-manifold $M$ with $\partial M \neq \nil$, a collection $\bfA$ of properly embedded arcs in $M$ is called a \emph{tunnel system of $M$} if $E(\bfA,M)$ is a handlebody. The \emph{tunnel number of $M$}, denoted by $t(M)$, is the minimum of $|\bfA|$ over all tunnel systems $\bfA$ of $M$. For a knot $K$ in a closed manifold, the tunnel number $t\big(E(K)\big)$ is precisely the well-studied \emph{tunnel number of $K$}, denoted by $t(K)$. The tunnel number $t(M)$ is a measure of the complexity of $M$, and it is a natural generalization of the tunnel number $t(K)$.

Now, given a tunnel system $\bfA$ of $M$ as above, $V:=\overline{N}(\bfA \cup \partial M)$ is a compression body with a minimal spine $\mathbf{\Sigma}_V=\bfA \cup \partial M$, and $W:=E(\mathbf{\Sigma}_V ,M) \cong E(\bfA, M)$ is a handlebody; hence, writing $S=\partial_+ V=\partial_+W$, we have a tunnel-type Heegaard splitting $M=V \cup_S W$. Conversely, for any tunnel-type Heegaard splitting $M=V \cup_S W$, any minimal spine $\mathbf{\Sigma}_V=\bfA \cup \partial M$ gives rise to a tunnel system $\bfA$. One may observe that, although there are an infinite number of choices for a tunnel system $\bfA$ corresponding to a given splitting $M=V \cup_S W$, the number of arcs in $\bfA$ is always given by $|\bfA|=g(S)-g(\partial M)+|\partial M|-1$, which depends only on $g(S)$ and not on a particular choice of a tunnel-system $\bfA$ for the splitting. Passing to the minimum over all tunnel systems on the left-hand side and over all tunnel-type splittings on the right-hand side, we have $t(M)=g^t(M)-g(\partial M)+|\partial M|-1$. In particular, for a knot $K$ in a closed manifold, the above formula yields $t(K)=t\big(E(K)\big)=g^t\big(E(K)\big)-1=g\big(E(K)\big)-1$.

From the above formula, it is clear that the tunnel number $t(M)$ and the tunnel-type genus $g^t(M)$ encode essentially the same information about the complexity of $M$, which plays a crucial role in our work. The tunnel-type genus often turns out to be a more convenient way to express this complexity for our purpose.

\subsection{Haken--Casson--Gordon Lemma} 

One of the most fundamental tools in the modern studies of Heegaard splittings is the Haken--Casson--Gordon Lemma \cite{Haken:SomeResults} \cite{Casson--Gordon:ReducingHeeg} and its consequences. Hayashi and Shimokawa \cite[Theorem 1.3]{Hayashi--Shimokawa:ThinPosition} gave a generalization of the lemma, which allows the presence of a link.

\begin{thm}[{\cite{Hayashi--Shimokawa:ThinPosition}}] \label{HCGHS}
Let $M$ be a connected $3$-manifold with non-empty boundary, and let $L$ be a (possibly empty) link in $M$. Let $M=V \cup_S W$ be a Heegaard splitting of $M$, and suppose that, if $L \neq \nil$, $L$ is in a bridge position with respect to $S$. If $D \subset E(L)$ is a compression disk for $\partial_-W$, then there exists a compression disk $D' \subset E(L)$ such that
\begin{enumerate}
\item $S \cap D'$ is an essential simple closed curve on $S$;
\item $D'$ is obtained from $D$ by $2$-surgeries and isotopy in $E(L)$;
\item if $E(L)$ is irreducible, then $D'$ is isotopic to $D$ in $E(L)$; and
\item there exist minimal disk systems of $V$ and $W$ disjoint from $D'$.
\end{enumerate}
\end{thm}

\begin{rem}
The conclusions (3) and (4) were not explicitly stated in \cite{Hayashi--Shimokawa:ThinPosition}. We note that (3) readily follows from the proof of (1) and (2) in \cite{Hayashi--Shimokawa:ThinPosition}. Also, the argument in \cite{Casson--Gordon:ReducingHeeg-Preprint}, where a statement analogous to (4) was given, easily carries over to our setting with minor modification and yields the conclusion (4).
\end{rem}

Following Kobayashi and Nishi \cite[Proposition 6.2]{Kobayashi--Nishi:Bing}, we can now describe how the tunnel-type genus of a compact $3$-manifold $M$ changes after cutting $M$ along a compression disk for the boundary.

\begin{prop}[c.f. \cite{Kobayashi--Nishi:Bing}] \label{compression}
Let $M$ be a $3$-manifold with non-empty boundary, and let $L$ be a (possibly empty) link in $M$ such that $E(L)$ is irreducible. Suppose $D \subset E(L)$ is a compression disk for $\partial M$, and let $M_-:=E(D,M)$ be the manifold obtained by cutting $M$ along $D$.
Then,
\[
g^t(M_-)=\begin{cases}
g^t(M) & \text{if $|M_-|=|M|+1$, or} \\
g^t(M)-1 & \text{if $|M_-|=|M|$.}
\end{cases}
\]
\end{prop}

\begin{proof}
Since the compression only affects the component of $M$ that contains $D$, we may assume $M$ is connected as well. Consider a minimal genus tunnel-type Heegaard splitting $M=V \cup_S W$ where $V$ is a handlebody. Isotoping $L$ if necessary, we may assume that $L$ is in a bridge position with respect to $S$. By Theorem~\ref{HCGHS}, we may isotope $D$ in $E(L)$ so that it meets $S$ in a simple closed essential curve on $S \cap E(L)$. The rest of the proof is identical to that of \cite{Kobayashi--Nishi:Bing}, and we omit it.
\end{proof}

\section{Irreducible Knots and Simple Knots} \label{Irreducible Knots and Simple Knots}

A knot $K$ in a closed connected $3$-manifold $M$ is said to be \emph{irreducible} if the exterior $E(K)$ is irreducible, and it is said to be \emph{simple} if $E(K)$ is irreducible and atoroidal. Simple knots comprise an important class of knots in the study of $3$-manifolds, since Thurston's hyperbolization theorem implies that the complement of a simple knot admits a unique cpmplete hyperbolic metric. In this section, we record some facts and observations about these knots and their exterior.

\subsection{Existence}

Irreducible knots and simple knots are known to exist in every closed connected $3$-manifold $M$. 
Generalizing Bing's ideas in \cite{Bing:S3}, Myers proved the existence of simple knots \cite[Theorem 6.1]{Myers:SimpleKnots}; the construction starts with a special handle decomposition of $M$ and produces a simple knot by connecting sufficiently complicated tangles in $0$-handles by arcs in $1$-handles that are parallel to the core arcs. Refining Myers' work, Hass and Thompson showed that this construction could produce an infinite number of such knots in $M$ that are pairwise distinct \cite[Proposition 3]{Hass--Thompson:Bing}.

\begin{thm}[{\cite{Myers:SimpleKnots}}, {\cite{Hass--Thompson:Bing}}] \label{existence-simple}
Let $M$ be a closed connected $3$-manifold. Then, there exist an infinite number of distinct simple knots in $M$.
\end{thm}

\begin{rem}
Using the theory of Heegaard splittings as the main tool, Rieck gave an alternative proof of the existence of irreducible knots in \cite{Rieck:Bing}.
\end{rem}

\subsection{Spherical and Toric Boundary}
If $K$ is an irreducible knot or a simple knot which is contained in a connected submanifold $Y$, some components of of $\overline{M-Y}$ can be described concretely. Specifically, we are concerned with a submanifold $Y$ such that $\partial Y$ contains an $S^2$-component or $T^2$-component.

First, let $K$ be an irreducible knot in a closed connected $3$-manifold $M$. The following lemma, involving an $S^2$-component of $\partial Y$, is immediate from the definition.

\begin{lem} \label{sphere boundary}
Let $M$ be a closed connected $3$-manifold and $Y$ be a connected $3$-submanifold of $M$. If there exists a knot $K \subset Y$ which is irreducible in $M$, then each $S^2$-component of $\partial Y$ bounds a closed $3$-ball outside $Y$.
\end{lem}

\begin{proof}
Let $S$ be an $S^2$-component of $\partial Y$. If a knot $K \subset Y$ is irreducible in $M$, $S \subset E(K)$ must bound a closed $3$-ball in $E(K)$. Such a ball must sit outside $Y$, since $K \subset Y$.
\end{proof}

Next, we consider $T^2$-components of $\partial Y$. If there exists a compression disk $D$ for $\partial Y$ such that $\partial D$ lies on a $T^2$-component of $\partial Y$, then compression along $D$ gives a sphere where we can utilize the irreducibility of $K$. Extending the arguments that have appeared in \cite[\S9]{Myers:SimpleKnots} and \cite{Hass--Thompson:Bing}, we record the following lemma.

\begin{lem} \label{torus boundary 0}
Let $M$ be a closed connected $3$-manifold and $Y$ be a connected $3$-submanifold of $M$. If there exists a knot $K \subset Y$ which is irreducible in $M$, and if there exists a compression disk $D \subset E(K)$ for $\partial Y$ such that $\partial D$ lies on some torus component $T \subset \partial Y$, then $T$ bounds a submanifold $Z$ outside of $Y$ such that $Z$ is homeomorphic to the exterior of some knot in $S^3$. Moreover, there exists a Fox reimbedding of $Y$ into $M$, restricting to the identity map on $K \subset Y$.
\end{lem}

\begin{proof}
Let $S$ denote the $2$-sphere obtained by compressing $T$ along $D$. Since $S$ bounds a ball in $E(K)$ by Lemma~\ref{sphere boundary}, we see that the original torus $T$ must bound a submanifold $Z$. Let us look at $Z$ more closely by considering two cases separately.

If $D \subset \overline{M-Y}$, compression along $D$ removes $N\big(D,\overline{M-Y}\big)$ from $\overline{M-Y}$. By the irreducibility of $E(K)$, it follows from Lemma~\ref{sphere boundary} that the $2$-sphere $S$ must bound a closed $3$-ball $B$ outside $Y \cup \overline{N}\big(D,\overline{M-Y}\big)$, and hence outside $Y$. Reversing the compression, we see that $T$ bounds a solid torus $Z:=B \cup \overline{N}\big(D, \overline{M-Y}\big)$.

If $D \subset Y$, compression along $D$ removes $N\big(D,E(K) \cap Y\big)$ from $E(K) \cap Y$; here, $N\big(D,E(K) \cap Y\big)$ can be regarded as a $1$-handle in $Y$ with $D$ as the cocore. Again by the irreducibility of $E(K)$, it follows from Lemma~\ref{sphere boundary} that the $2$-sphere $S$ must bound a closed ball $B$ such that $B \cap Y$ is precisely the $1$-handle that we removed from $Y$. Reversing the compression, we recover $Y$ by adding this $1$-handle back, which simultaneously drills out the regular neighborhood of the core of this $1$-handle from $B$ to yield $Z$. By construction, $Z$ is the exterior of some knot in $S^3$. 

To prove the last assertion, we only need to modify the last step in the $D \subset Y$ case above. Instead of recovering $Y$ by adding the $1$-handle $N\big(D,E(K) \cap Y\big)$ back to $Y-N\big(D,E(K) \cap Y\big)$, we can add a trivial $1$-handle inside $B$ to $Y-N\big(D,E(K) \cap Y\big)$. We obtain a submanifold $Y'$ homeomorphic to $Y$, and the new torus boundary now bounds a solid torus outside $Y'$; in other words, $Y$ reimbeds as $Y'$ with the desired properties.
\end{proof}

Now, suppose that $K$ is a simple knot in a closed connected $3$-manifold $M$ so that every $T^2$-component $T \subset \partial Y$ is compressible in $E(K)$ unless it is $\partial$-parallel in $E(K)$. So, if we can assure that $T$ is not $\partial$-parallel in $E(K)$, then Lemma~\ref{torus boundary 0} can be applied to $T$. Indeed, this is essentially how Hass and Thompson argue in \cite{Hass--Thompson:Bing}, after choosing a simple knot $K$ such that $T$ is not $\partial$-parallel in $E(K)$; for their purpose, it was sufficient to pick two simple knots from an infinite collection of simple knots, since one of them must have the exterior in which $T$ is not $\partial$-parallel.

We would like to have a better control on the choice of a simple knot $K$, that allows further generalization. As it turns out, we can prevent the torus $T \subset \partial Y$ from becoming $\partial$-parallel in $E(K)$ by imposing a lower bound on the tunnel number of $K$. For now, let us illustrate this idea by giving the following lemma, which treats the case where $\partial Y$ consists of a single $T^2$-component. We will give a generalization that allows higher genera and multiple components in \S\ref{Characteristic Knots}.

\begin{lem} \label{torus boundary 1}
Let $M$ be a closed connected $3$-manifold and $Y$ be a connected $3$-submanifold of $M$ with $\partial Y \cong T^2$, and let $Z:=\overline{M-Y}$. If there exists a knot $K \subset Y$ which is simple in $M$ with $t(K) \geq g(Z)$, then $Z$ is homeomorphic to the exterior of some knot in $S^3$. Moreover, there exists a Fox reimbedding of $Y$, restricting to the identity map on $K \subset Y$.
\end{lem}

\begin{proof}
If $\partial Y$ is a $\partial$-parallel torus in $E(K)$, $Y$ must be a solid torus with $K$ as its core. Then, it is easy to see that $g(Z)=g\big(E(K)\big)=t(K)+1>t(K)$, which contradicts our assumption. Thus, by the definition of simple knots, $\partial Y \subset E(K)$ must be a compressible torus in $E(K)$. The conclusion follows from Lemma~\ref{torus boundary 0}.
\end{proof}

\section{Amalgamated Heegaard Genus} \label{Amalgamated Heegaard Genus}

When a connected $3$-manifold $M$ is separated by a closed surface $F$ into two submanifolds $Y_1$ and $Y_2$, there is a natural construction of a Heegaard splitting of $M$ from that of $Y_1$ and $Y_2$, say $Y_1=V_1 \cup_{S_1} W_1$ and $Y_2=V_2 \cup_{S_2} W_2$ with $F=\partial_-V_1 \cap \partial_-V_2$. In the process, loosely speaking, one combines $V_1$ and $W_2$ on one side, $V_2$ and $W_1$ on the other side, by collapsing the collar neighborhood of $F$. This construction, called \emph{amalgamation} along $F$, has its origin in the work of Casson and Gordon \cite{Casson--Gordon:ReducingHeeg}, and was first defined explicitly by Schultens in \cite{Schultens:TrivialSurfaceBundle}. 

In this section, we study the genus of Heegaard splittings amalgamated along $F$. In particular, we introduce the notion of the \emph{amalgamated Heegaard genus} of $M$ with respect to $F$, and study some of its properties. The amalgamated Heegaard genus measures the complexity of the position of $F$ inside $M$, and it will be used extensively in the next section to control the reimbedding process.

\subsection{Amalgamated Splitting}

Let $M$ be a closed connected $3$-manifold and let $F \subset M$ be a (possibly disconnected) closed separating surface such that $M=Y_1 \cup_F Y_2$ with $F=\partial Y_1=\partial Y_2$. Suppose we have tunnel-type Heegaard splittings $Y_i=V_i \cup_{S_i} W_i$ for each $i$, i.e. $V_i$ is a union of compression bodies with $\partial_-V_i=\partial Y_i=F$, and $W_i$ is a union of handlebodies.

Fix the product structure on $\overline{N}(F, Y_i)$ by a homeomorphism $F \times I \cong \overline{N}(F, Y_i)$, sending $F \times \{0\} \subset F \times I$ to $F \subset \overline{N}(F, Y_i)$. By definition, $V_i$ is obtained by attaching $1$-handles to $\overline{N}(F, Y_i)$ so that the ends of these $1$-handles are glued onto the image of $F \times \{1\}$. One can extend these $1$-handles vertically through the product region $\overline{N}(F,Y_i)$ so that they are attached to $F$. After doing so with the 1-handles in both $V_1$ and $V_2$, perturbing if necessary so that the attaching disks for handles from opposite sides are disjoint on $F$, we take the union of the boundaries of all $1$-handles together with $F$ and then remove the interior of attaching disks. One can check that this yields a closed connected surface which is indeed a Heegaard surface for $M$. The corresponding Heegaard splitting is said to be the \emph{amalgamation of Heegaard splittings $Y_i=V_i \cup_{S_i} W_i$ along $F$}.

\subsection{Amalgamated Genus}

Given a separating surface $F$, say $M=Y_1 \cup_F Y_2$ as before, one may consider all possible Heegaard splittings for $Y_i$ and the corresponding Heegaard splittings of $M$ obtained by amalgamation along $F$. It is natural, then, to consider the following.

\begin{defn}
For a closed 3-manifold $M$ and a separating surface $F$, we define the \emph{amalgamated Heegaard genus of $M$ with respect to $F$}, denoted by $g(M;F)$, to be the minimum of genera of Heegaard splittings of $M$ obtained by amalgamation along $F$.
\end{defn}

 If each component of $F$ is separating, the genus of this amalgamated Heegaard surface is $g(S_1)+g(S_2)-g(F)$. Hence, it is easy to see that the amalgamated Heegaard genus $g(M;F)$ is minimized if and only if we amalgamate the minimal genus tunnel-type splittings of $Y_1$ and $Y_2$. Thus, assuming that each component of $F$ is separating, we have a formula
\[
g(M;F)=g^t(Y_1)+g^t(Y_2)-g(F).
\]
As Rieck pointed out to us (see \cite[\S2]{Kobayashi--Rieck:ConnSum}), if some components of $F$ are not separating, the genus of the Heegaard surface obtained by amalgamating the Heegaard splittings $Y_i=V_i \cup_{S_i} W_i$ is not $g(S_1)+g(S_2)-g(F)$. For example, suppose $M=Y_1 \cup_F Y_2$ where each $Y_i$ is connected and $F$ consists of $m>1$ components, so that every component of $F$ is non-separating; in this case, the genus of the Heegaard surface obtained by amalgamation is $g(S_1)+g(S_2)-g(F)+(m-1)$, and hence the amalgamated Heegaard genus is given by the formula
\[
g(M;F)=g^t(Y_1)+g^t(Y_2)-g(F)+(m-1).
\]
As it turns out, in the proof of our main theorem where we utilize the notion of amalgamated Heegaard genus, each component of the separating surface $F$ is separating.

\subsection{Properties}

Starting with a separating surface $F$, one often obtains a new separating surface by modifying $F$ by standard procedures such as taking subsurface or compressing along a disk. We describe how the amalgamated Heegaard genus changes under such modifications.

\begin{lem} \label{multi to single}
Let $M$ be a closed connected $3$-manifold and $F \subset M$ be a closed surface that separates $M$ into $Y$ and $Z$ such that $Y$ is connected.
If $F$ consists of components $F_1, \cdots, F_n$, such that each $F_i$ separates $M$, then we have $g(M;F_i) \leq g(M;F)$ for each $i$.
\end{lem}

\begin{proof}
Since $Y$ is connected and each $F_i$ is separating, it follows that each $F_i$ bounds a component of $Z$. It follows that $Z$ consists of $n$ components, and we can denote them as $Z_1, \cdots, Z_n$ such that $\partial Z_i=F_i$ for each $i$.

For each $i$, let us write $Z^*_i:=Z-Z_i$ and $F^*_i=F-F_i$ so that $\partial Z^*_i=F^*_i$. Also let $Y_i=\overline{M-Z_i}=Y \cup_{F^*_i} Z^*_i$, so that $\partial Y_i=F_i$ and $M=Z_i \cup_{F_i} Y_i$. Note that the amalgamation of tunnel-type Heegaard splittings of $Y$ and $Z^*_i$ along $F^*_i$ produces a Heegaard splitting of $Y_i$. Starting with the minimal genus tunnel-type splittings of $Y$ and $Z^*_i$, we have
\[
g^t(Y_i) \leq g^t(Y)+g^t(Z^*_i)-g(F^*_i) = g^t(Y)+\sum_{j \neq i} g^t(Z_j)-\sum_{j \neq i} g(F_j).
\]
Thus, we obtain
\begin{eqnarray*}
g(M;F_i) & = &g^t(Y_i)+g^t(Z_i)-g^t(F_i) \\
& \leq & \bigg( g^t(Y)+\sum_{j \neq i} g^t(Z_j)-\sum_{j \neq i} g(F_j) \bigg)+g^t(Z_i)-g(F_i) \\
& = & g^t(Y)+\sum_j g^t(Z_j)-\sum_j g(F_j) \\
& = & g^t(Y)+g^t(Z)-g(F) \\
& = & g(M;F).
\end{eqnarray*}
\end{proof}

The next Lemma, which may be of independent interests, is stated without an assumption that each component of $F$ is separating. In the proof of our main theorem, we will only need the case where each component of $F$ is separating.

\begin{lem} \label{compressed AHG}
Let $M$ be a closed connected $3$-manifold and $F \subset M$ be a closed connected surface that separates $M$. Suppose $L \subset M-F$ is a (possibly empty) link in $M$ such that $E(L)$ is irreducible. If $F$ is compressible in $E(L)$ and $F_\circ$ is the surface obtained by compressing $F$ along a compression disk in $E(L)$, then $F_\circ$ also separates $M$ and $g(M;F_\circ) \leq g(M;F)$.
\end{lem}

\begin{proof}
Suppose $M=Y \cup_F Z$, i.e. $F$ separates $M$ into submanifolds $Y$ and $Z$. Clearly, the surface $F_\circ$ obtained by compressing the separating surface $F$ is also separating. Without loss of generality, assume that there is a compression disk $D \subset E(L) \cap Y$ for $F$. Let $Y_-:=Y-N\big(D,E(L) \cap Y\big)$ and $Z_+:=\overline{M-Y_-}=Z \cup \overline{N}\big(D,E(L) \cap Y\big)$, so that we can set $F_\circ:=\partial Y_-=\partial Z_+$. Since any Heegaard surface of $Z$ can be regarded as a tunnel-type Heegaard surface of the connected submanifold $Z_+$, we have $g^t(Z_+) \leq g^t(Z)$.

Suppose $\partial D$ lies on $F$ as a separating loop, so that $g(F_\circ)=g(F)$. If $D$ lies in $Y$ as a separating disk, we have $g^t(Y_-) = g^t(Y)$ by Proposition~\ref{compression}. Thus, we obtain
\[
\begin{array}{rcccl}
g(M;F_\circ) &=& g^t(Y_-)+g^t(Z_+)-g(F_\circ) &&\\
&=& g^t(Y)+g^t(Z_+)-g(F) &&\\
&\leq& g^t(Y)+g^t(Z)-g(F) &=& g(M;F).
\end{array}
\]
If $D$ lies in $Y$ as a non-separating disk, we have $g^t(Y_-) = g^t(Y)-1$ by Proposition~\ref{compression}. In this case, note that $F_\circ$ consists of two components, and each of these components is non-separating. Using the formula in the earlier remark, we obtain
\[
\begin{array}{rcccl}
g(M;F_\circ) &=& g^t(Y_-)+g^t(Z_+)-g(F_\circ)+1 &&\\
&=& \big(g^t(Y)-1\big)+g^t(Z_+)-g(F)+1 &&\\
&=& g^t(Y)+g^t(Z_+)-g(F) &&\\
&\leq& g^t(Y)+g^t(Z)-g(F) &=& g(M;F).
\end{array}
\]
Now, suppose $\partial D$ lies on $F$ as a non-separating loop, so that $g(F_\circ)=g(F)-1$. In this case, $D$ lies in $Y$ as a non-separating disk, and we have $g^t(Y_-)=g^t(Y)-1$ by Proposition~\ref{compression}. Thus, we obtain
\[
\begin{array}{rcccl}
g(M;F_\circ) &=& g^t(Y_-)+g^t(Z_+)-g(F_\circ) &&\\
&=& \big(g^t(Y)-1 \big)+g^t(Z_+)-\big(g(F)-1\big) &&\\
&=& g^t(Y)+g^t(Z_+)-g(F) &&\\
&\leq& g^t(Y)+g^t(Z)-g(F) &=& g(M;F).
\end{array}
\]
In all cases, we have established the inequality $g(M;F_\circ) \leq g(M;F)$ as desired.
\end{proof}

\begin{rem}
In the proof of Lemma~\ref{compressed AHG} above, we have shown that
\[
\begin{array}{rcccl}
g(M;F_\circ) &=& g^t(Y)+g^t(Z_+)-g(F) &&\\
&\leq& g^t(Y)+g^t(Z)-g(F) &=& g(M;F)
\end{array}
\]
holds for all cases. Hence, it follows that the equality $g(M;F_\circ) = g(M;F)$ occurs if and only if $g^t(Z_+)=g^t(Z)$.
\end{rem}

\section{Characteristic Knots} \label{Characteristic Knots}

Let $M$ be a closed connected $3$-manifold. Generalizing the notions of irreducible knots and simple knots in a natural way, one can study a class of knots $K$ whose exterior $E(K)$ contains no closed orientable incompressible surface of genus at most $g$ except for the $\partial$-parallel torus. In the terminology of Kobayashi and Nishi \cite{Kobayashi--Nishi:Bing}, such knots are said to be \emph{$g$-characteristic in $M$}. Irreducible knots and simple knots in a closed connected $3$-manifold are precisely $0$-characteristic knots and $1$-characteristic knots respectively. In this section, we will combine the results from previous sections to establish Fox reimbedding statements for submanifolds that contain characteristic knots. The proof of our main theorem will be given in the next section as an immediate consequence of these propositions.

\subsection{Existence}

For any integer $g \geq 0$, Kobayashi and Nishi \cite{Kobayashi--Nishi:Bing} established the existence of a $g$-characteristic knots in a closed connected manifold $M$. They further showed the following refinement, stated here in terms of the tunnel number of knots.

\begin{thm}[{\cite[Theorem 5.1]{Kobayashi--Nishi:Bing}}] \label{existence-char}
Let $M$ be a closed connected $3$-manifold. Then, for any integers $g \geq 0$ and $t \geq 0$, there exist an infinite number of g-characteristic knots $K \subset M$ with $t(K) \geq t$.
\end{thm}

\begin{rem}
If a knot $K$ is $g$-characteristic for all $g\geq0$, $K$ is said to be a \emph{small} knot. The existence of small knots in a closed $3$-manifold is not known in general. It is conjectured that every non-Haken manifolds contains a small knot. The existence of small knots for some special cases can be found in \cite{Lopez:SmallKnots}, \cite{Matsuda:SmallKnots}, and \cite{Qiu--Wang:SmallKnots}.
\end{rem}

\subsection{Surfaces in the Exterior}

A closed orientable surface $F$ of small genera in the exterior of $g$-characteristic knots share some properties with closed orientable surfaces in irreducible non-Haken $3$-manifolds. One similarity, by the definition of $g$-characteristic knots, is the compressibility of $F$: if $g(F) \leq g$, then $F$ is compressible unless $F$ is a $\partial$-parallel torus. Another similarity, which is indeed a consequence of compressibility of $F$, is given in the next lemma.

\begin{lem} \label{separating}
Fix an integer $g \geq 0$, and let $M$ be a closed connected $3$-manifold. If there exists a knot $K \subset M$ which is $g$-characteristic in $M$, then every connected orientable surface $F \subset E(K)$ with $g(F) \leq g$ must separate $M$.
\end{lem}

\begin{proof}
Suppose there is a non-separating connected orientable surface $F \subset E(K)$ with $g(F) \leq g$. Since $K$ is $g$-characteristic, $F$ must either be compressible or $\partial$-parallel in $E(K)$. Since $F$ is clearly separating in the latter case, we may assume that $F$ is compressible in $E(K)$. Compressing $F$ along a compression disk $D \subset E(K)$ and keeping a non-separating component, we obtain a non-separating connected orientable surface with strictly smaller genus. The process must eventually end with a incompressible surface $F'$ with $g(F') \leq g$, which contradicts the assumption.
\end{proof}

Suppose that $Y$ is a connected $3$-submanifold of $M$, and such that $Y$ contains a knot $K$ which is $g$-characteristic in $M$. In this situation, Lemma~\ref{separating} implies that a component $F \subset \partial Y$ with $g(F) \leq g$ must bound a submanifold $Z$ outside $Y$ such that $\partial Z=F$.

\subsection{Connected Boundary}

The next proposition generalizes the reimbeddability statement of Lemma~\ref{torus boundary 1} to the cases where $\partial Y$ is a connected surface of higher genus. The key assumption is that the tunnel number of the characteristic knot is bounded by the amalgamated genus.

\begin{prop} \label{single boundary}
Let $M$ be a closed connected $3$-manifold and $Y$ be a connected $3$-submanifold of $M$ with non-empty connected boundary $F:=\partial Y$. If there exists a knot $K \subset Y$ which is $g(F)$-characteristic in $M$ with $t(K) \geq g(M;F)$, then there exists a Fox reimbedding of $Y$ into $M$, restricting to the identity map on $K \subset Y$.
\end{prop}

\begin{proof}
If $g(F)=0$, it follows from Lemma~\ref{sphere boundary} that $F$ bounds a closed $3$-ball outside $Y$. In other words, $Y$ is already a Fox submanifold, and no reimbedding is necessary. For $g(F) \geq 1$, we prove the proposition by induction on $g(F)$. Let $Z:=\overline{M-Y}$ so that $F=\partial Y=\partial Z$. For the base case, suppose $g(F)=1$. Then, $t(K) \geq g(M;F)=g^t(Y)+g^t(Z)-g(F) \geq 1+g(Z)-1=g(Z)$, and the statement follows from Lemma~\ref{torus boundary 1}.

Now, for the induction, let $g>1$ and suppose that the statement holds when $1 \leq g(F) < g$; we aim to show that the statement also holds when $g(F)=g$. By assumption, there exists a knot $K \subset Y$ which is $g(F)$-characteristic in $M$ with $t(K) \geq g(M;F)$. Since $g(F)=g>1$, $F$ must be compressible in $E(K)$. We consider the following four cases separately.

\begin{case} 
\emph{There is a compression disk $D \subset Z$ for $F$, such that its boundary $\partial D$ does not separate $F$.}
\vspace{2mm}

Let $Y_+$ be the manifold obtained from $Y$ by adding a $2$-handle $Q=\overline{N}(D,Z)$. Let $Z_-:=\overline{M-Y_+}$ and $F_\circ:=\partial Y_+=\partial Z_-$. Since $F_\circ$ is the surface obtained by compressing $F$ along $D$, $F_\circ$ is a connected surface with $g(F_\circ)=g(F)-1=g-1 \geq 1$.

We already know that the knot $K \subset Y_+$ is $(g-1)$-characteristic in $M$. Also, by Lemma~\ref{compressed AHG}, $g(M;F_\circ) \leq g(M;F) \leq t(K)$. Thus, invoking the induction hypothesis, there exists a reimbedding of $Y_+$, which restricts to the identity map on $K \subset Y_+$, with the image $Y'_+ \subset M$ so that $Z'_-:=\overline{M-Y'_+}$ is a handlebody. Note that we have a decomposition $Y'_+=Y' \cup Q'$ where $Y' \supset K$ and $Q'$ are the images of $Y \supset K$ and $Q$ respectively. Moreover, $\overline{M-Y'}=Z'_- \cup Q'$ is a handlebody, since $Q'$ is added to the handlebody $Z'_-$ as a $1$-handle. One can now easily obtain a reimbedding of $Y$ onto $Y' \cong Y$ with the desired properties.
\end{case}

\begin{figure}[t]
\label{fig:case-1}
\begin{center} \includegraphics[height=74mm, width=94mm]{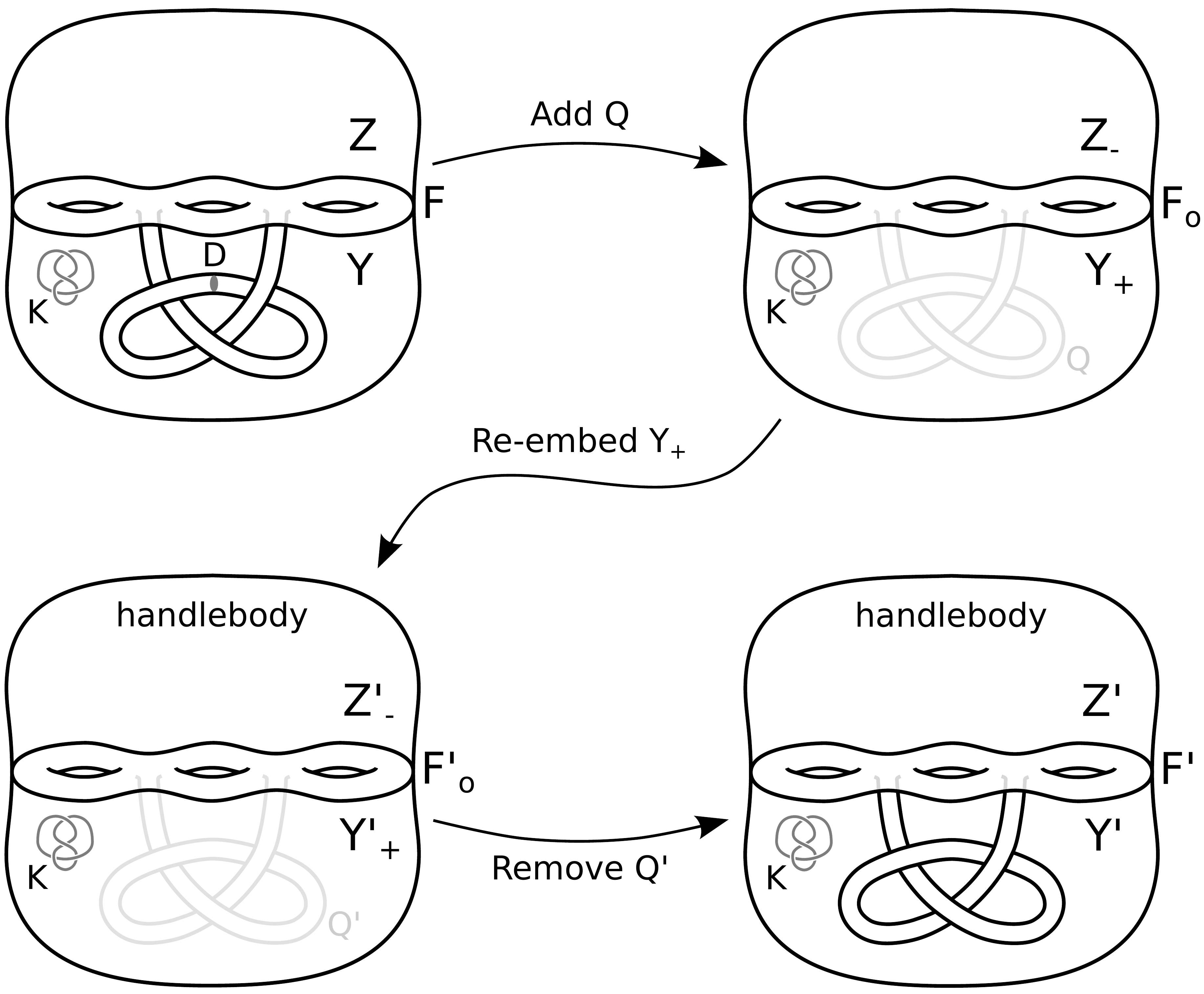} \end{center}
\caption{Schematic pictures for Case 1.}
\end{figure}

\begin{figure}[b]
\label{fig:case-2}
\begin{center} \includegraphics[height=74mm, width=94mm]{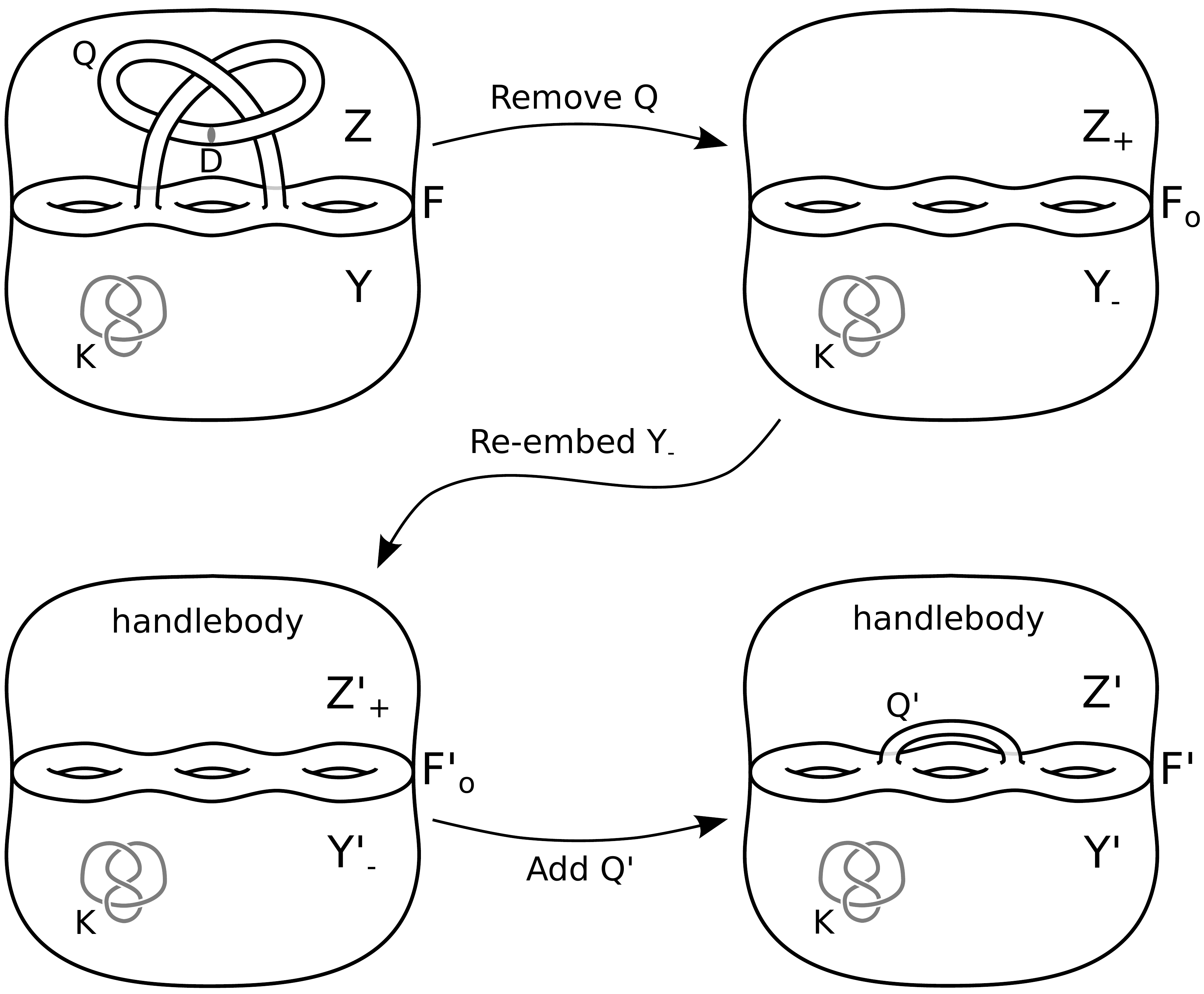} \end{center}
\caption{Schematic pictures for Case 2.}
\end{figure}

\begin{case} 
\emph{There is a compression disk $D \subset E(K) \cap Y$ for $F$, such that its boundary $\partial D$ does not separate $F$.}
\vspace{2mm}

Let $Z_+$ be the manifold obtained from $Z$ by adding a $2$-handle $Q=\overline{N}\big(D, E(K) \cap Y \big)$. Let $Y_-:=\overline{M-Z_+}$ and $F_\circ:=\partial Y_-=\partial Z_+$. By the argument analogous to Case 1 above, we can invoke the induction hypothesis and obtain a reimbedding of $Y_-$, which restricts to the identity map on $K \subset Y_-$, with the image $Y'_- \subset M$ so that $Z'_+:=\overline{M-Y'_-}$ is a handlebody.

Now, let $\alpha$ be a properly embedded $\partial$-parallel arc in $Z'_+$. Then, $Y':=Y'_- \cup \overline{N}(\alpha,Z'_+) \supset K$ is a homeomorphic copy of $Y$, and $\overline{M-Y'}=Z'_+-N(\alpha,Z'_+)$ is a handlebody. One can now easily obtain a reimbedding of $Y$ onto $Y' \cong Y$ with the desired properties.
\end{case}

\begin{case} 
\emph{There is a compression disk $D \subset Z$ for $F$, such that its boundary $\partial D$ separates $F$.}
\par \vspace{2mm}

Let $Y_+$ be the manifold obtained by adding a $2$-handle $Q=\overline{N}(D,Z)$ along the neighborhood of a separating essential loop $\partial D$ on $F$. Let $Z_-:=\overline{M-Y_+}$ and $F_\circ:=\partial Y_+=\partial Z_-$. Since $F_\circ$ is the surface obtained by compressing $\partial Y$ along $D$, $F_\circ$ is a two-component surface, say $F_\circ=F_{\circ,1} \sqcup F_{\circ,2}$, with $g(F_{\circ,1})+g(F_{\circ,2})=g(F)$ and $g(F) > g(F_{\circ,i}) \geq 1$. Note that, by Lemma~\ref{separating}, each $F_{\circ,i}$ is a separating surface. Thus, one sees that $D$ must have been a separating disk for $Z$, and that $Z_-$ consists of two components. Write these components as $Z_{-,1}$ and $Z_{-,2}$ so that $\partial Z_{-,i}=F_{\circ,i}$.

Let $Y_{+,1}:=Y_+ \cup_{F_{\circ,2}} Z_{-,2}$, so that $\partial Y_{+,1}=F_{\circ,1}=\partial Z_{-,1}$. We already know that the knot $K \subset Y_+ \subset Y_{+,1}$ is $g(F_{\circ,1})$-characteristic. Also, by Lemma~\ref{multi to single} and Lemma~\ref{compressed AHG}, $g(M;F_{\circ,1}) \leq g(M;F_\circ) \leq g(M;F) \leq t(K)$. Thus, invoking the induction hypothesis, there exists a reimbedding of $Y_{+,1}$, restricting to the identity map on $K \subset Y_{+,1}$, with the image $Y'_{+,1} \subset M$ so that $Z'_{-,1}:=\overline{M-Y'_{+,1}}$ is a handlebody. Let us write $F'_{\circ,1}$ for the image of $F_{\circ,1}$ so that $\partial Y'_{+,1}=F'_{\circ,1}=\partial Z'_{-,1}$. Note that we have a decomposition $Y'_{+,1}=Y'_+ \cup_{F'_{\circ,2}} Z'_{-,2}$, where $Y'_+$, $F'_{\circ,2}$, and $Z'_{-,2}$ are the images of $Y_+$, $F_{\circ,2}$, and $Z_{-,2}$ respectively. We also note that $K \subset Y'_+$ since $K \subset Y_+$ before the reimbedding.

\begin{figure}[ht]
\label{fig:case-3}
\begin{center} \includegraphics[height=129mm, width=114mm]{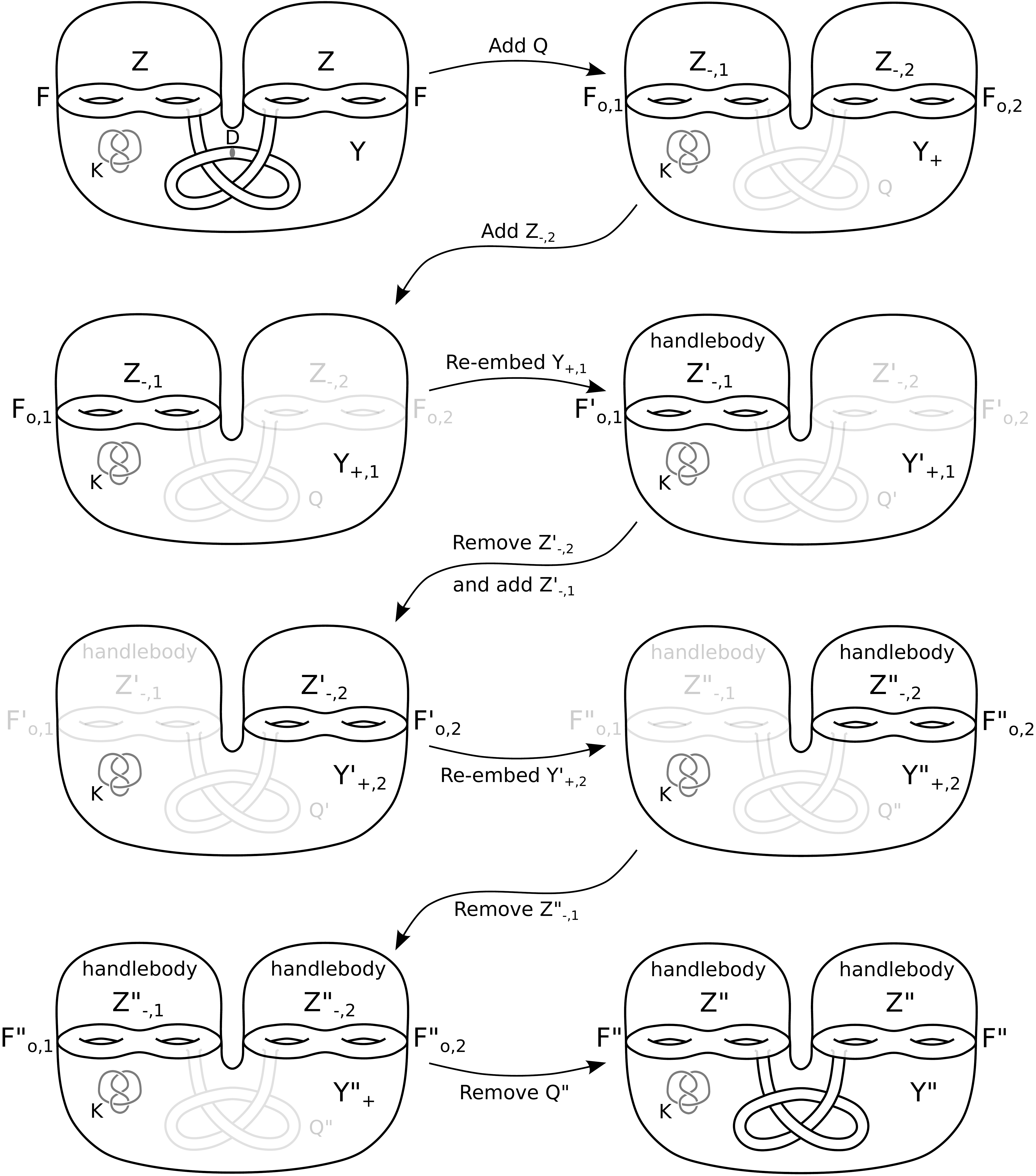} \end{center}
\caption{Schematic pictures for Case 3.}
\end{figure}

Let $Y'_{+,2}:=Y'_+ \cup_{F'_{\circ,1}} Z'_{-,1}$ so that $\partial Y'_{+,2}=F'_{\circ,2}=\partial Z'_{-,2}$. In order to repeat the argument in the last paragraph, let us first verify the necessary inequality to invoke the induction hypothesis. Writing $F'_\circ=F'_{\circ,1} \sqcup F'_{\circ,2}$, we have $F'_\circ \cong F_\circ$ and thus $g(F'_\circ)=g(F_\circ)$. Also, since $Y'_+ \cong Y_+$ and $Z'_{-,2} \cong Z_{-,2}$, we have $g^t(Y'_+)=g^t(Y_+)$ and $g^t(Z'_{-,2})=g^t(Z_{-,2})$. On the other hand, since $Z'_{-,1}$ is a handlebody, $g^t(Z'_{-,1})=g(F'_{\circ,1})=g(F_{\circ,1}) \leq g^t(Z_{-,1})$. Thus, we have
\begin{eqnarray*}
g(M;F'_\circ) &=& g^t(Y'_+)+g^t(Z'_-)-g(F'_\circ) \\
&=& g^t(Y'_+)+g^t(Z'_{-,1})+g^t(Z'_{-,2})-g(F'_\circ) \\
&\leq& g^t(Y_+)+g^t(Z_{-,1})+g^t(Z_{-,2})-g(F_\circ) \\
&=& g^t(Y_+)+g^t(Z_-)-g(F_\circ) \\
&=& g(M;F_\circ).
\end{eqnarray*}
Together with Lemma~\ref{multi to single} and Lemma~\ref{compressed AHG}, we now obtain
\[
g(M;F'_{\circ,2}) \leq g(M;F'_\circ) \leq g(M;F_\circ) \leq g(M;F) \leq t(K),
\]
which is the desired inequality for the induction.

Of course, $K \subset Y'_+ \subset Y'_{+,2}$ is $g(F'_{\circ,2})$-characteristic, since $K$ is $g(F)$-characteristic by assumption and $g(F'_{\circ,2})=g(F_{\circ,2}) \leq g(F)$. Thus, invoking the induction hypothesis, there exists a reimbedding of $Y'_{+,2}$, restricting to the identity map on $K \subset Y'_{+,2}$, with the image $Y''_{+,2}$ so that $Z''_{-,2}:=\overline{M-Y''_{+,2}}$ is a handlebody. Note that we have a decomposition $Y''_{+,2}=Y''_+ \cup_{F''_{\circ,1}} Z''_{-,1}$, where $Y''_+$, $F''_{\circ,1}$, and $Z''_{-,1}$ are the images of $Y'_+$, $F'_{\circ,1}$, and $Z'_{-,1}$ respectively. In particular, we have a decomposition $M=Y''_+ \cup Z''_{-,1} \cup Z''_{-,2}$, where each $Z''_{-,i}$ is a handlebody. We also have a decomposition $Y''_+=Y'' \cup Q''$ where $Y'' \supset K$ and $Q''$ are images of $Y \supset K$ and $Q$ respectively under the composition of the two reimbeddings. The cocore of the $2$-handle $Q''$ is an arc connecting $F''_{\circ,1}$ and $F''_{\circ,2}$; thus, $Q''$ can be regarded as a $1$-handle that connects handlebodies $Z''_{-,1}$ and $Z''_{-,2}$. It follows that $\overline{M-Y''}=Z''_{-,1} \cup Q'' \cup Z''_{-,2}$ is a handlebody. One can now easily obtain a reimbedding of $Y$ onto $Y'' \cong Y$ with the desired properties.
\end{case}

\begin{case} 
\emph{Every compression disk $D \subset E(K)$ for $F$ is contained in $E(K) \cap Y$ and its boundary $\partial D$ separates $F$.}
\par \vspace{2mm}
We compress $F$ into $E(K) \cap Y$ along a compression disk, and continue compressing the resulted surface, if possible, into the remnant of $E(K) \cap Y$ along one disk at a time. Isotoping the compression disk in each step if necessary, we may assume that the boundary of each compression disk misses the pairs of disks from previous compressions and hence lies in the original $F$. Then, it follows that each disk in this process is indeed a compression disk for the original $F$ into $E(K) \cap Y$. The sequence of such compressions must stop, since the maximal number of $2$-surgeries that we can apply to $F$ along separating loops is at most $g(F)-1$. Let us write $D_1, \cdots, D_n$ for these compression disks, and let $\bfD:=D_1 \sqcup \cdots \sqcup D_n\subset E(K) \cap Y$. 

Take $\bfQ:=\overline{N}\big(\bfD,E(K) \cap Y\big)$, and let $Z_+:=Z \cup \bfQ$, $Y_-:=\overline{M-Z_+}$, and $F_\circ:=\partial Y_-=\partial Z_+$. By construction, $F_\circ$ is the surface obtained from $F$ by compression along $\bfD$; by construction, it admits no further compression into the remnant of $E(K) \cap Y$, and it has no $S^2$-component. It also follows that $F_\circ$ consists of $n+1$ components, say $F_\circ=F_{\circ,1} \sqcup \cdots \sqcup F_{\circ,n+1}$ with $\sum_{i=1}^{n+1} g(F_{\circ,i})=g(F)$ and $g(F) > g(F_{\circ,i}) \geq 1$. Note that, by Lemma~\ref{separating}, each $F_{\circ,i}$ is a separating surface. Thus, one sees that each disk in $\bfD$ must have been a separating disk for $Y$, and that $Y_-$ consists of $n+1$ components. Let $Y_-=Y_{-,1} \sqcup \cdots \sqcup Y_{-,n+1}$ so that $\partial Y_{-,i}=F_{\circ,i}$. We remark that $Y$ is the boundary connect sum of components of $Y_-$.

\begin{claim*}
$F_\circ$ is compressible into $Z_+$.
\end{claim*}

\begin{proof}[Proof of Claim]
First, suppose that $F_\circ$ consists of $\partial$-parallel tori in $E(K)$. Note that $Z$ must be connected, since $M$ is connected and $Z \subset M$ is a submanifold with connected boundary $\partial Z=F$; hence, $Z_+$ is also connected. Then, together with the assumption that $F_\circ=\partial Z_+$ consists of $\partial$-parallel tori, it follows that $Z_+$ must be a parallelism, i.e. $Z_+ \cong T^2 \times I$. Let $Y_-=Y_{-,1} \sqcup Y_{-,2}$ with $K \subset Y_{-,1}$. Note that $Y_{-,1} \cup Z_+$ is a solid tori with the knot $K$ as its core. Thus, we have $Y_{-,2} \cong E(K)$, and hence
\[
t(K)=g^t\big(E(K)\big)-1=g^t(Y_{-,2})-1 < g^t(Y_{-,2}).
\]
On the other hand, we have
\begin{eqnarray*}
t(K) & \geq & g(M;F) \\
& \geq & g(M;F_\circ) \\
&=&g^t(Y_-)+g^t(Z_+)-g(F_\circ) \\
&=&g^t(Y_{-,1})+g^t(Y_{-,2})+g^t(Z_+)-g(F_{\circ,1})-g(F_{\circ,2}) \\
&=&1+g^t(Y_{-,2})+2-1-1 \\
&=&g^t(Y_{-,2})+1 \\
&>&g^t(Y_{-,2}),
\end{eqnarray*}
where the first inequality is an assumption of the theorem and the second inequality follows from Lemma~\ref{compressed AHG}. The two inequalities above clearly contradicts.

Hence, we may assume that there is a component of $F_\circ$ which is not a $\partial$-parallel torus in $E(K)$ and hence compressible in $E(K)$. We can find a compression disk for $F_\circ$ by the standard argument as follows. Let $D$ be a compression disk for a compressible component of $F_\circ$. If $D \cap F_\circ$ contains trivial loops in $F_\circ$, we may surger $D$ along the disks bounded by such loops and obtain a new compression disk which intersects $F_\circ$ only in essential loops; so, we may as well assume that $D \cap F_\circ$ consists of essential loops in $F_\circ$. We may also assume that $D \cap F_\circ=\partial D$; if not, replace $D$ with a subdisk of $D$ bounded by an innermost component of $D \cap F$ in $D$. With these two assumptions, $D$ is indeed a compression disk for $F_\circ$. Finally, since $F_\circ$ is not compressible into $Y_-$ by construction, $F_\circ$ is compressible into $Z_+$.
\end{proof}

Let $D$ be a compression disk of $F_\circ$ into $Z_+$, as in the above claim. If we reconnect components of $Y_-$ by 1-handles that avoids $D$, we obtain a submanifold $Y' \cong Y$, i.e. a reimbedding of $Y$, such that this disk $D$ is a compression disk for $\partial Y'$ lying outside $Y'$. We will show that, if we choose these 1-handles carefully, we can retain the necessary condition on the amalgamated Heegaard genus and apply Case 1 or Case 3 to $Y'$.

Let $Z=V \cup_S W$ be a minimal genus Heegaard splitting with $\partial_-V=\partial Z$ and $\partial_- W=\nil$, i.e. $W$ is a handlebody. Then, $V_+:=V \cup \bfQ$ is a compression body with $\partial_-V_+=\partial Z_+$ and $\partial_+V_+=\partial_+ V=S=\partial W$. It follows that $Z_+=V_+ \cup_S W$ is a tunnel-type Heegaard splitting of $Z_+$.

By the version of Haken's Lemma due to Casson and Gordon (c.f. Theorem~\ref{HCGHS}), we can isotope $D$ so that $D \cap S$ is a simple closed essential curve in $S$ and find a minimal disk system $\bfE$ for $V_+$ such that $\bfE \cap D=\nil$. Note that the dual spine $\bfA$ of $\bfE$ is a minimal spine consisting of properly embedded arcs in $V_+$, and it is also disjoint from $D$.

\begin{figure}[h]
\label{fig:case-4}
\begin{center} \includegraphics[height=62mm, width=124mm]{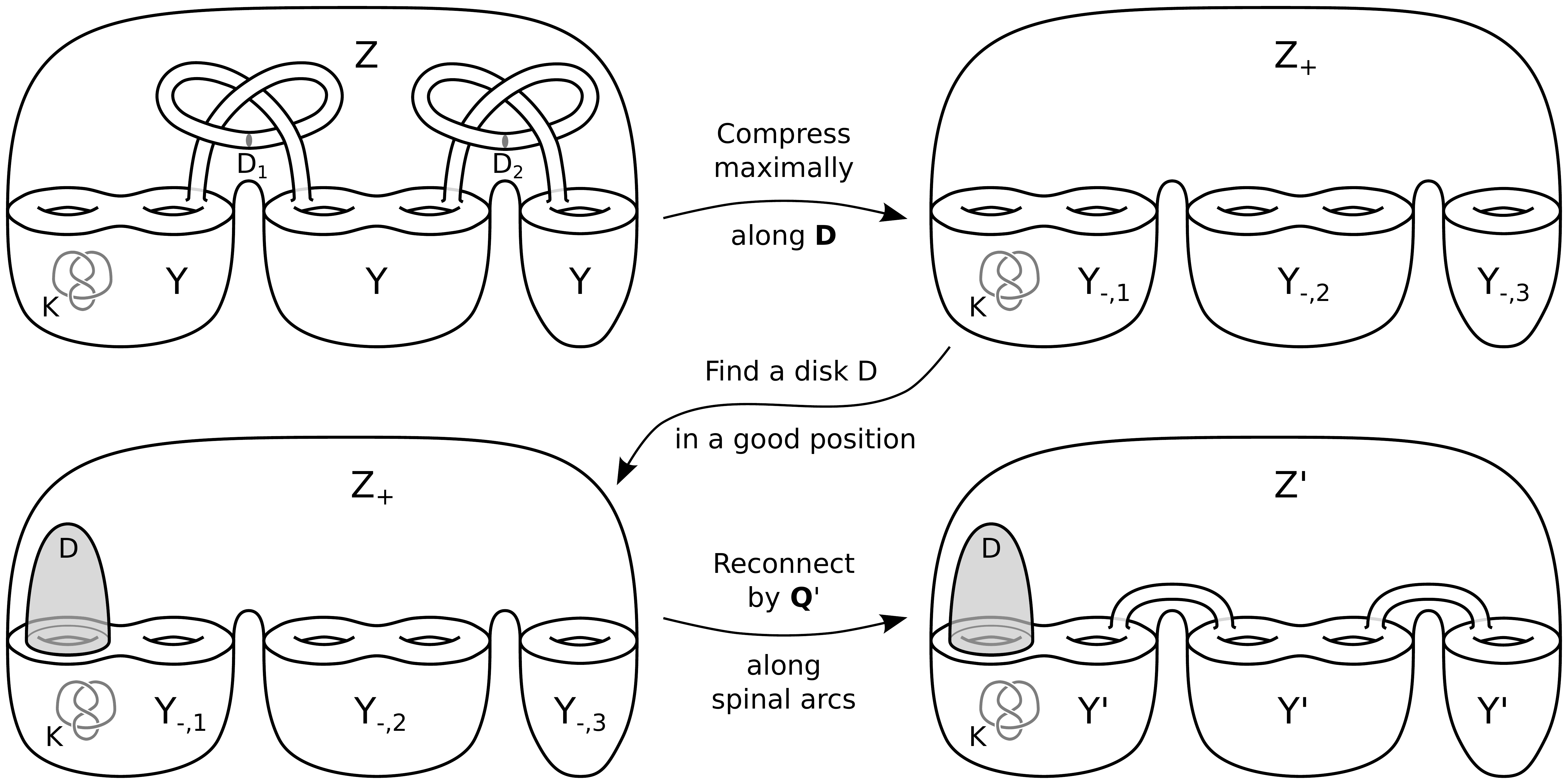} \end{center}
\caption{Schematic pictures for Case 4.}
\end{figure}

One can choose a subcollection $\bfA'$ of the spine $\bfA$, such that (i) $\bfA' \cup \partial_-V_+$ is connected, and (ii) the number of arcs in $\bfA'$ is minimal among subcollections that satisfy (i). Since $\partial_-V_+$ consists of $n+1$ components, one can see that $\bfA'$ consists of precisely $n$ arcs. Now, take $\bfQ':=\overline{N}(\bfA,V_+-D)$, and let $Y':=Y_- \cup \bfQ'$. Also, let $Z':=\overline{M-Y'}$ and $F':=\partial Y'=\partial Z'$.

$Y'$ is obtained from $Y_-$ by re-connecting the components of $Y_-$ by $1$-handles $\bfQ'$ along the arcs $\bfA'$. It follows from the choice of $\bfA'$ that $Y'$ is indeed the boundary connect sum of components of $Y_-$; thus, $Y'$ is homeomorphic to $Y$. Moreover, the compression disk $D$ can now be regarded as a compression disk for $F'$ into $Z'$ as desired. Note that one can easily construct the reimbedding of $Y$ with image $Y'$ so that it restricts to the identity map on $K$.

We observe that $Z'$ can be decomposed as $Z'=V' \cup_S W$, where $V':=\overline{V_+ -\bfQ'}$. Since $V'$ is obtained from the compression body $V_+$ by drilling out $\bfQ'$ along the subcollection of spinal arcs, we see that $V'$ is again a compression body with $\partial_+V'=\partial_+V_+=S=\partial W$. In other words, the decomposition $Z'=V' \cup_S W$ is indeed a Heegaard splitting for $Z'$, and we have $g^t(Z') \leq g(S) =g^t(Z)$. Hence, together with $g^t(Y')=g^t(Y)$ and $g(F')=g(F)$, we obtain
\[
\begin{array}{rcccccl}
g(M;F') &=& g^t(Y')+g^t(Z')-g(F') &&&&\\
&\leq& g^t(Y)+g^t(Z)-g(F) &=& g(M;F) &\leq& t(K).
\end{array}
\]
Since $F'=\partial Y'$ compresses into $Z'=\overline{M-Y'}$ along $D$, we can now apply Case 1 or Case 3 to $Y'$ and obtain the desired reimbedding of $Y$.
\end{case}
\noindent These four cases together complete the proof of Proposition~\ref{single boundary}.
\end{proof}

\subsection{Disconnected Boundary}
In Proposition~\ref{single boundary}, the submanifold $Y$ has a single boundary component. We now give the analogous reimbedding statement for the case where $Y$ has a multiple boundary components.

\begin{prop} \label{multiple boundary}
Let $M$ be a closed connected $3$-manifold and $Y$ be a connected $3$-submanifold of $M$ with non-empty boundary $F:=\partial Y$. If there exists a knot $K \subset Y$ which is $g_{\max}(F)$-characteristic in $M$ with $t(K) \geq g(M;F)$, then there exists a Fox reimbedding of $Y$ into $M$, restricting to the identity map on $K \subset Y$.
\end{prop}

\begin{proof}
We prove the statement by induction on the number of components $|F|$ of $F$. When $|F|=1$, i.e. $F$ is connected, the statement coincides with Proposition~\ref{single boundary}. For the induction, let $m > 1$ and suppose that the statement holds when $1 \leq |F| < m$; we aim to show that the statement also holds for the case $|F|=m$.

Let us write $F_i$ for each component of $F$ so that $F:=F_1 \sqcup \cdots \sqcup F_m$, and let $Z:=\overline{M-Y}$ as before. Since $K \subset Y$ is $g_{\max}(F)$-characteristic with $g_{\max}(F) \geq g(F_i)$, we know from Lemma~\ref{separating} that each $F_i$ must be a separating surface. It follows that each $F_i$ bounds a component of $Z$ outside $Y$. Let $Z=Z_1 \sqcup \cdots \sqcup Z_m$ with $\partial Z_i=F_i$.

Take $Y_1:=Y \cup Z_2 \cup \cdots \cup Z_m$, so that $\partial Y_1=\partial Z_1=F_1$. By the argument analogous to Case 3 in the proof of Proposition~\ref{single boundary}, we can invoke induction hypothesis to obtain a reimbedding of $Y_1$ onto $Y'_1$, restricting to the identity on $K$, so that the image $F'_1$ of $F_1$ bounds a handlebody $Z'_1$ outside $Y'_1$. Note that $Y'_1$ decomposes as $Y'_1:=Y' \cup Z'_2 \cup \cdots \cup Z'_m$, where $Y', Z'_2, \cdots, Z'_m$ are the images of $Y, Z_2, \cdots, Z_m$ respectively. For $i=2, \cdots, m$, the image $F'_i$ of $F_i$ bounds $Z'_i$ outside $Y'$.

Now, take $Y'_0:=Y' \cup_{F'_1} Z'_1$, $\partial Y'_0=\bigsqcup_{i=2}^m F'_i$. Again by the argument analogous to Case 3 in the the proof of Proposition~\ref{single boundary}, one can invoke the induction hypothesis to obtain a reimbedding of $Y'_0$ onto $Y''_0$, restricting to the identity on $K$, so that the image $F''_i$ of $F'_i$ bounds a handlebody $Z''_i$ for $i=2, \cdots, m$. Note that $Y''_0$ decomposes as $Y''_0=Y'' \cup Z''_1$, where $Y''$ and $Z''_1$ are the images of $Y'$ and $Z'_1$ respectively. The image $F''_1$ of $F'_1$ also bounds the handlebody $Z''_1$.

Thus, composing the two reimbeddings, we have a reimbedding of $Y$ onto $Y''$, restricting to the identity on $K$, so that each component $F''_i$ of $F''=\partial Y''$ bounds a handlebody.
\end{proof}

\section{Main Theorem and Examples} \label{Main Theorem and Examples}

We can now complete the proof of our main theorem that relates Fox submanifolds and Bing submanifolds of a closed connected $3$-manifold. We also collect some special cases, which include some known results and a few new corollaries.

\subsection{Main Theorem}

Most of the work is already done in the last section, and the theorem is an immediate consequence of Proposition~\ref{multiple boundary}.

\begin{thm4}
For any closed connected $3$-manifold $M$, every Bing submanifold of $M$ admits a Fox reimbedding into $M$; hence, a compact connected 3-manifold can be embedded in $M$ as a Fox submanifold if and only if it can be embedded in $M$ as a Bing submanifold.
\end{thm4}

\begin{proof}
Let $Y$ be a Bing submanifold of a closed connected $3$-manifold $M$. By Theorem~\ref{existence-char}, there exists a $g_{\max}(\partial Y)$-characteristic knot $K \subset M$ with $t(K) \geq g(M;\partial Y)$. Since $Y$ is a Bing submanifold of $M$, we may isotope $K$ into $Y$; thus, we may as well assume $K \subset Y$. By Proposition~\ref{multiple boundary}, we can reimbed $Y$ as a Fox submanifold $Y' \cong Y$.
\end{proof}

\subsection{Special Cases}

It is worthwhile to record the following examples, which show that some of the theorems mentioned in \S\ref{Introduction} are indeed special cases of Theorem~\ref{main theorem}.

\begin{exa}
Suppose $M \cong S^3$. It is easy to see that every submanifold of $S^3$ is a Bing submanifold. Thus, our Theorem~\ref{main theorem} specializes to the original Fox reimbedding theorem \cite{Fox:Reimbedding}, which we cited as Theorem~\ref{Fox} in \S\ref{Introduction}.
\end{exa}

\begin{exa}
Suppose $Y \cong B^3$. Since every embedding of a closed $3$-ball into a connected $3$-manifold is isotopic to each other, Theorem~\ref{main theorem} says that the closure of the complement of $Y$ is a single closed $3$-ball as well. In other words, $M \cong S^3$. This is essentially Bing's characterization of $S^3$ \cite{Bing:S3}, which we cited as Theorem~\ref{Bing} in \S\ref{Introduction}.
\end{exa}

\begin{exa}
Suppose $Y$ is a genus $g$ handlebody. Theorem~\ref{main theorem} says that there is another genus $g$ handlebody, say $Y'$ such that the closure of the complement $Y'$ is a handlebody of genus $g$ as well. In other words, $M=Y' \cup \overline{M-Y'}$ is a Heegaard splitting of $M$. This is essentially the main results of \cite{Hass--Thompson:Bing} and \cite{Kobayashi--Nishi:Bing}, since the other direction of these theorems is again trivial.
\end{exa}

The latter two examples together provide the characterization of closed connected $3$-manifolds that admit genus $g$ splittings for each $g \geq 0$. We can also give the following alternative characterization of such manifolds as an immediate consequence of Theorem~\ref{main theorem}.

\begin{coro} \label{Heeg}
A closed connected $3$-manifold $M$ admits a genus $g$ Heegaard splitting if and only if there exists a closed orientable surface $F$ in $M$ with genus $g$ such that every knot in $M$ can be isotoped to lie within a regular neighborhood of $F$.
\end{coro}

\begin{proof}
The statement follows immediately if we apply Theorem~\ref{main theorem} to the neighborhood $Y:=\overline{N}(F,M)$.
\end{proof}

Note that we can apply the same idea to closed non-orientable surfaces. Recall from \cite{Rubinstein:One-Sided} that, for a closed connected $3$-manifold $M$, a pair $(M,F)$ is said to be a \emph{one-sided Heegaard splitting of $M$} if $F$ is a closed non-orientable surface such that $E(F,M)$ is a handlebody. The cross-cap genus of this one-sided splitting is defined to be that of $F$. The statement analogous to Corollary~\ref{Heeg} in this context is the following.

\begin{coro}
A closed connected $3$-manifold $M$ admits a cross-cap genus $g$ one-sided Heegaard splitting if and only if there exists a closed non-orientable surface $F$ in $M$ with cross-cap genus $g$ such that every knot in $M$ can be isotoped to lie within a regular neighborhood of $F$.
\end{coro}

\begin{proof}
Again, the statement follows immediately if we apply Theorem~\ref{main theorem} to the neighborhood $Y:=\overline{N}(F,M)$.
\end{proof}

\subsection{Homotopy vs. Isotopy} \label{Homotopy}

With the affirmative resolution of the \emph{Geometrization Conjecture}, particularly the \emph{Poincar\'{e} Conjecture} and the \emph{Spherical Space Form Conjecture} in three-dimension, the following statements are now known to hold.
\begin{itemize}
\item A closed connected $3$-manifold $M$ is homeomorphic to the $3$-sphere if and only if every loop in $M$ can be \emph{homotoped} to lie within a closed $3$-ball.
\item A closed connected $3$-manifold $M$ is homeomorphic to a lens space (possibly $S^3$ and $S^2 \times S^1$) if and only if there exists a solid torus $V \subset M$ such that every loop in $M$ can be \emph{homotoped} to lie within $V$.
\end{itemize}
Note that, although significantly stronger and tremendously more difficult to establish, these statements are formally similar to the characterization of $S^3$ by Bing~\cite{Bing:S3} and the characterization of lens spaces by Hass and Thompson~\cite{Hass--Thompson:Bing} respectively. The difference is that the above statements allow homotopy of loops while the statements in \cite{Bing:S3} and \cite{Hass--Thompson:Bing} use isotopy of knots.

It is then natural to ask if the statement of our Theorem~\ref{main theorem} remain valid with such a modification as well, possibly as a consequence of the resolution of the Geometrization Conjecture.

\begin{defn}
Let $M$ be a connected 3-manifold, and $Y$ be a compact connected submanifold of $M$. $Y$ is said to be a \emph{$\pi_1$-surjective submanifold} if the inclusion $i: Y \hookrightarrow M$ induces a surjection $i_*:\pi_1(Y) \twoheadrightarrow \pi_1(M)$; equivalently, $Y$ is said to be a \emph{$\pi_1$-surjective submanifold} if every loop in $M$ can be homotoped to lie within $Y$.
\end{defn}

\noindent As noted above, by the resolution of the Geometrization Conjecture, a $\pi_1$-surjective closed $3$-ball is always a Fox submanifold and a $\pi_1$-surjective solid torus always admits a Fox reimbedding. We address the following question: for any closed connected $3$-manifold $M$, does every $\pi_1$-surjective submanifold of $M$ admit a Fox reimbedding into $M$? 

\begin{prop} \label{counterexample}
There exists a closed connected $3$-manifold $M$ and a $\pi_1$-surjective submanifold $Y$ of $M$, such that $Y$ admits no Fox reimbedding into $M$. Moreover, there are infinite number of such pairs $M$ and $Y$.
\end{prop}

The proposition gives the negative answer to our question above. \emph{A priori}, there is no reason for our Theorem~\ref{main theorem} to generalize to $\pi_1$-injective submanifolds. However, as the above discussion suggests, finding an example of $M$ and $Y$ as in Proposition~\ref{counterexample} can be a rather delicate task.

We present one family of examples that has already been studied in literature. For a $3$-manifold $M$, the \emph{rank} $r(M)$ is defined to be the smallest cardinality of generating set for $\pi_1(M)$. If $M$ is a closed $3$-manifold that admits a genus $g$ Heegaard splitting, then a handlebody in the splitting is $\pi_1$-surjective, and it gives rise to a generating set for $\pi_1(M)$ with $g$ elements; hence, from a minimal genus splitting in particular, it follows that the inequality $r(M) \leq g(M)$ always hold. Observing $r(M)=g(M)$ for many examples, Waldhausen \cite{Waldhausen:Problems} asked if this equality holds in general.

This question was answered negatively by Boileau and Zieschang in \cite{Boileau--Zieschang:HeegaardGenus}, where they found Seifert fibered manifolds $M$ with $r(M)=2$ and $g(M)=3$. The work of Weidmann in \cite{Weidmann:2-Generated} provides further examples of closed manifolds $M$ with $r(M)<g(M)$. Schultens and Weidmann showed in \cite{Schultens--Weidmann:GeomAlgRank} that, for any non-negative integer $n$, there exists a closed connected $3$-manifold $M$ with $r(M)<g(M)=r(M)+n$. Furthermore, very recently, Li annoounced in \cite{Li:RankGenus} that there exist hyperbolic $3$-manifolds $M$ that satisfies $r(M)<g(M)$.

\begin{proof}[Proof of Proposition~\ref{counterexample}]
Let $M$ be a closed $3$-manifold with $r(M)<g(M)$; as already mentioned, an infinite number of such manifolds are known from \cite{Schultens--Weidmann:GeomAlgRank}. Choose a base point $* \in M$, and let $\Gamma$ be the union of simple closed curves with the base point $*$, representing a minimal generating set of $\pi_1(M,*)$. Homotoping the curves if necessary, we can choose $\Gamma$ so that $\Gamma$ is a bouquet of loops and hence $Y:=\overline{N}(\Gamma,M)$ is a $\pi_1$-surjective handlebody of genus $r(M)$.

If $Y$ admits a Fox reimbedding into $M$, the reimbedded image and its exterior are both genus $r(M)$ handlebodies, giving a genus $r(M)$ Heegaard splitting $M$. Then, we have $g(M) \leq r(M)$, which clearly contradicts our choice of $M$.
\end{proof}

The proof of Proposition~\ref{counterexample} above suggests a relationship between Waldhausen's question and our Theorem~\ref{main theorem}. Refining the above arguments slightly, we give a characterization of $3$-manifolds that satisfy $r(M) < g(M)$.

\begin{prop}
Let $M$ be a closed $3$-manifold. Then, the following are equivalent:
\begin{itemize}
\item $r(M)<g(M)$;
\item For any $\pi_1$-surjective handlebody $V \subset M$ of rank $r(M)$, there is a knot $K$ in $M$ that cannot be isotoped to lie within $V$.
\end{itemize}
\end{prop}

\begin{proof}
Suppose $r(M) < g(M)$ and let $V \subset M$ be a $\pi_1$-surjective handlebody of rank $r(M)$. Let $K \subset M$ be a $r(M)$-caracteristic knot with $t(K) \geq g(M;\partial V)$; the existence of such a knot is given by Theorem~\ref{existence-char}, i.e. \cite[Theorem 5.1]{Kobayashi--Nishi:Bing}. If $K$ can be isotoped to lie within $V$, then $V$ admits a Fox re-imbedding by Proposition~\ref{single boundary}. As in the proof of Proposition~\ref{counterexample}, any Fox reimbedding of $V$ yields a Heegaard splitting of genus $r(M)$, and we have $r(M) \geq g(M)$ which contradicts the assumption $r(M) < g(M)$. So, this knot $K$ cannot be isotoped to lie within $V$.

For the other direction, the contrapositive statement easily holds. If $r(M) \geq g(M)$, and hence $r(M)=g(M)$, take a handlebody $V$ of genus $g(M)$ in a minimal genus Heegaard splitting of $M$; $V$ is a Bing submanifold, and any knot $K \subset M$ can be isotoped to lie within $V$.
\end{proof}


\bibliography{Bing-Fox}
\bibliographystyle{amsalpha}

\end{document}